\documentclass[11pt, reqno]{amsart}
\usepackage[utf8]{inputenc}
\usepackage{amsmath}
\usepackage{amsthm}
\usepackage{amsfonts}
\usepackage{amssymb}
\usepackage{mathrsfs, mathtools, mathdots}
\usepackage{diagbox}
\usepackage{cancel}

\newtheorem{theorem}{Theorem}[section]
\newtheorem{lemma}{Lemma}[section]

\newtheorem{corollary}{Corollary}[section]
\newtheorem{definition}{Definition}[section]

\usepackage{hyperref}
\usepackage[top=30mm,bottom=30mm,left=30mm,right=30mm]{geometry}

\begin{document}

\title{An effective estimate for the sum of two cubes problem}
\author{Saunak Bhattacharjee}
\address{Indian Institute of Science Education and Research, Tirupati \\ Andhra Pradesh,
India – 517619. }
\email{saunakbhattacharjee@students.iisertirupati.ac.in}
\date{}

\begin{abstract}
 Let $f(x, y) \in \mathbb{Z}[x, y]$ be a cubic form with non-zero discriminant and for each integer $m \in \mathbb{Z}$,
let $N_{f}(m)=\#\left\{(x, y) \in \mathbb{Z}^{2}: f(x, y)=m\right\} $.
In 1983, Silverman proved that there is a positive constant $c_0$ such that $N_{f}(m)>c_0\left(\log |m|\right)^{3 / 5}$ holds for infinitely many integers $m$.
In this paper, we obtain an effective estimate for $N_f(m)$, namely, showing that if $f(x, y)=x^{3}+y^{3}$, then $N_{f}(m)>4.2\times 10^{-6}(\log |m|)^{11/13}$ holds for infinitely many $m \in \mathbb Z$.
\end{abstract}


\maketitle

\section{\bf Introduction}

\bigskip

\noindent
Let $f(x, y) \in \mathbb{Z}[x, y]$ be a cubic form with non-zero discriminant. For each integer $m \in \mathbb{Z}$, define
$$
N_{f}(m)=:\#\left\{(x, y) \in \mathbb{Z}^{2}: f(x, y)=m\right\} .
$$

\medskip
\noindent
It has been a topic of interest to study how large the size of $N_f(m)$ can be. 

\medskip
\noindent
In this direction, the first significant result is due to Chowla. In $1933$, Chowla \cite{c5} proved that if $f(x,y)=x^3 - ky^3$, where $k$ is a nonzero integer, then
there is a positive constant $c$ such that the following inequality holds for infinitely many positive integers $m$
$$
N_{f}(m)> c \, \left(\log \log m \right).
$$

\medskip
\noindent
In $1935$, Mahler \cite{c6} improved Chowla's result by showing that for every cubic form $f(x,y)$ with nonzero discriminant, there exists a constant $c_f>0$, which depends on $f$ such that
\begin{equation}
    \label{1}
N_{f}(m)> c_f \, \left(\log |m|\right)^{1 / 4}
\end{equation}
\noindent
holds for infinitely many integers $m$.

\medskip

\noindent
Invoking the theory of height functions, Silverman \cite{c1} extended the idea of Chowla and Mahler. This resulted in an improvement of the exponent to $1/3$ in (\ref{1}) and simplification of the calculations as well. Further in $2008$, Stewart \cite{c7} improved this exponent to $1/2$. Chowla, Mahler, Silverman, Stewart, all of their ideas connected this problem to the construction of elliptic curves of high rank, and it turns out that one can do much better if we restrict $f(x,y)$ to be $x^3+y^3$.

\medskip

\noindent
More precisely, if $f(x, y)=x^{3}+y^{3},$ using a suitable elliptic curve of rank $3$, Silverman proved that there is a positive constant $c_0$ such that 
$$
N_{f}(m)>c_0\left(\log |m|\right)^{3 / 5}
$$ 
 holds for infinitely many integers $m$.
 
\medskip
\noindent
Moreover, in \cite{c7} Stewart mentions that the exponent can be improved to $11/13$ by using an elliptic curve of rank $11$ which was constructed by Elkies and Rogers \cite{c4} in $2004$.

\medskip

\noindent
It is to be noted that, in all the previous works of Chowla, Mahler, Silverman, and Stewart, the constant $c_f$ was not explicit. In this short note, using methods employed by Silverman and exploiting the properties of canonical height function on elliptic curves, we explicitly capture the dependence of the implied constant on the cubic form $f(x,y)=x^3+y^3$.

\bigskip

\begin{theorem}\label{Main theorem}
    Let $f(x, y)=x^{3}+y^{3}$ and $m_{0} \in \mathbb{Z}$ be a non zero integer. Consider the curve $E$ with homogeneous equation
$$
E: f(x, y)=m_{0}\, z^{3}
$$

\noindent
with the point $O=[1,-1,0]$ defined over $\mathbb{Q}$. Using $O$ as origin, we give $E$ the structure of an elliptic curve. Let
$$
r=\operatorname{rank} E(\mathbb{Q})
$$
be the rank of the Mordell-Weil group of $E / \mathbb{Q}$. Then for infinitely many integers $m$,
$$
N_f(m)> \frac{1}{(9\times2^{r+1}-20)^{r/r+2}(\hat{h}(\Bar{P}))^{r/r+2}}(\log |m|)^{r /(r+2)},
$$
where $\hat{h}(\Bar{P})$ denotes the canonical height of a specified point $\Bar{P}$ on $E': Y^2 = X^3-432m_{0}^{2}$, defined over $\mathbb{Q}$, with the base point $[0,1,0]$ at infinity.

\end{theorem}
\bigskip

\noindent
As an immediate corollary, we have the following.
\bigskip

\begin{corollary}\label{Main result}
     Let $f(x, y)=x^{3}+y^{3}$. Then there exist infinitely many integers $m$, such that
     $$N_{f}(m)>4.2\times 10^{-6}(\log |m|)^{11/13}$$
\end{corollary}

\bigskip
\section{\bf Preliminaries }
\bigskip

\noindent
We state and develop the necessary tools to prove our main theorem.

\medskip

\begin{lemma}\label{Lemma 1}
    Suppose there exist integers $x$, $y$, $z$ and $m_0$, such that $x^3+y^3=m_0z^3$ with $gcd(x,y,z)=1$, $gcd(12m_0z,x+y)=d$, then
    $|d|<3^{1/3}12m_0^{5/2}|z|^{1/2}$.
\end{lemma}
\begin{proof}

Suppose
\begin{center}
    $da=12m_0z$ and $db=x+y$, where $gcd(a,b)=1$.
\end{center}

\noindent
We want to show that $$d^2|3\times12^3m_0^2b.$$

\bigskip
\noindent
Let $p$ be a prime dividing $d$ and $p^k || d$, 
    we show that, $p^{2k}$ always divides $3\times 12^3m_0^2b$.

\medskip

\noindent
As $p^{k}|d\Rightarrow{p^k|12m_0z}$, there are two possible cases
\begin{center}
    $p^k|12m_0$ \hspace{10Pt}or \hspace{10Pt} $p|z$.
\end{center}
 If $p^k|12m_0$, then we are done. So, we only need to consider the case when $p|z$.\\

\noindent
Note that,
\begin{align*}
x^3+y^3 &= m_0z^3 \\
&\Rightarrow 12^3m_0^2(x+y)((x+y)^2-3xy) = 12^3m_0^3z^3 \\
&\Rightarrow 12^3m_0^2b(d^2b^2-3xy) = d^2a^3 \\
&\Rightarrow d^2|3 \times 12^3m_0^2bxy.
\end{align*}

\noindent
 Also,
\begin{center}
    $p|d \Rightarrow p|x+y$.
\end{center}

\noindent
 On the other hand, we have $p|z$ and $gcd(x,y,z)=1$, which implies that $p\nmid xy$.

\medskip
\noindent
As  $d^2|3 \times 12^3m_0^2bxy$, the primes dividing $d$ and not dividing $xy$, should be completely contained in the factorisation of $3\times12^3m_0^2b$.

\medskip
\noindent
Hence, $p^{2k}|3\times 12^3m_0^2b$ , for all prime $p$ dividing $d$ $\Rightarrow d^2|3\times12^3m_0^2b $.

\medskip
\noindent
This implies
\begin{align*}
& |d|^2\leq 3 \times 12^3m_0^2|b| \\
&\Rightarrow |d|^3\leq 3 \times 12^3m_0^2|b||d|= 3 \times 12^3m_0^2|x+y| \\
&\Rightarrow |d|^6\leq (3 \times 12^3m_0^2)^2(|x+y|)^2 \leq (3 \times 12^3m_0^2)^2 |x^3+y^3|= (3 \times 12^3m_0^2)^2m_0|z|^3 \\
&\Rightarrow |d|<3^{1/3}12m_0^{5/2}|z|^{1/2}.
\end{align*}

\end{proof}

\noindent
Now, we discuss some essential properties of heights on elliptic curves.

\begin{definition}
    The canonical height or N\'{e}ron-Tate height denoted by $\hat{h}$ on an elliptic curve $E/\mathbb Q$, is defined as the following 
    \medskip
    \begin{center}
        $\hat{h}(P) = \frac{1}{2} \lim_{{k \to \infty}} \left(4^{-k} h_x\left([2^k]P\right)\right)$, $P \in E(\overline {\mathbb Q})$,
    \end{center}
\medskip
    \noindent
    where $h_x(P)=\log \left(H\left([x(P),1]\right)\right)$ and $H$ is the height on $\mathbb P^1(\overline{\mathbb Q})$. 
\end{definition}
\noindent
For our purpose, we will only need the height $H$ on $\mathbb P^1(\mathbb Q)$, which is given by the following
\medskip
\begin{center}
    $H(P)= \text{max}\{|p|, \,\, |q|\},$
\end{center}
\medskip
\noindent
     where $P=[p,q] \in \mathbb P^1(\mathbb Q)$ is written in the reduced form such that $p, q \in \mathbb Z$ and $\text{gcd}(p,q)=1$.
     
\medskip
\noindent
For a more general definition of height on projective space, see \cite{c2}, pp. $226$.

\bigskip
\noindent
Next, we state a result due to N\'{e}ron \cite{neron} and Tate \cite{tate}, which gives the required properties of  canonical height.(See pp. $248-251$  \cite{c2}.)

\begin{lemma}[\textbf{N\'{e}ron, Tate}] \label{lemma 2}
    Let  $\hat{h}$ be the canonical height or N\'{e}ron-Tate height on an elliptic curve $E/\mathbb Q$, then
   
\medskip

\noindent
\begin{enumerate}
\item[$(a)$]  (Parallelogram law) For all $P,Q \in E(\overline{\mathbb Q})$ we have
    \[ \hat{h}(P + Q) + \hat{h}(P - Q) = 2\hat{h}(P) + 2\hat{h}(Q). \]
    
 \item[$(b)$] For all $P \in E(\overline{\mathbb Q})$ and all $m \in \mathbb{Z}$,
    \[ \hat{h}([m]P) = m^2\hat{h}(P). \]
\item[$(c)$] The canonical height $\hat{h}$ is a quadratic form on $E$, i.e., $\hat{h}$ is an even function, and the pairing
    \[ \langle \cdot, \cdot \rangle : E(\overline{\mathbb Q}) \times E(\overline{\mathbb Q}) \to \mathbb{R}, \quad \langle P,Q \rangle = \hat{h}(P + Q) - \hat{h}(P) - \hat{h}(Q), \]
    is bilinear.\\
    
\item[$(d)$] Let $P \in E(\overline{\mathbb Q})$. Then $\hat{h}(P) \geq 0$, and $\hat{h}(P) = 0$ if and only if $P$ is a torsion point.\\
    
\item[$(e)$] \[ \hat{h} = \frac{1}{2}h_x + O(1), \]
    where the $O(1)$ depends on $E$ and $x$. 
\end{enumerate}
\end{lemma}

\medskip

\noindent
Note that, in $(e)$ of Lemma \ref{lemma 2}, the implied constant is not explicit. In 1990, Silverman proved a general result which explicitly determines this constant in terms of the $j$-invariant and the discriminant of the elliptic curve. We invoke a specific case of this result here.(See [\cite{c3}, pp.$726$.])
\medskip

\begin{lemma}\label{lemma 3}
    Let, $E/\mathbb Q$ be an elliptic curve given by the Weierstrass equation 
    \begin{center}
        $y^2 = x^3 + B$ , then for every $P\in E(\Bar{\mathbb Q}) $
    \end{center}
\vspace{0.8 Pt}    
    \begin{center}
          $-\frac{1}{6}h(B)-1.48 \leq \hat{h}(P) - \frac{1}{2}h_x(P)\leq \frac{1}{6}h(B) + 1.576 $
    \end{center}
   
\end{lemma}

\bigskip

\noindent
Once we obtain the estimate of $N_f(m)$ in Theorem \ref{Main theorem}, the explicit result in Corollary \ref{Main result} follows from the existence of a specific elliptic curve of rank 11, which was constructed by Elkies and Rogers in \cite{c4}, as stated below.
\medskip

\begin{theorem}[\textbf{Elkies, Rogers}]\label{prop 1}
     The elliptic curve given by the equation
    \bigskip
    
    \begin{center}
         $x^3 + y^3 = m_0z^3$ or the Weierstrass form $Y^2=X^3-432m_0^2$,
    \end{center}
    
    \bigskip
       
where $m_0=13293998056584952174157235$  has the Mordell-Weil rank $11$.\\

\noindent
Moreover, $\max \{h_x(P_i) \mid 1\leq i  \leq 11 \}=76.61$, where $P_i$ varies over $11$ independent points of the Mordell-Weil group.
\end{theorem}
\noindent
    For the construction of this elliptic curve and the list of 11 independent points; see pp. 192-193 \cite{c4} . The rest follows by simple computation.

\bigskip

\section{\bf Proof of Theorem \ref{Main theorem} and Corollary \ref{Main result}}
\bigskip

\noindent
Consider the elliptic curve $E/\mathbb{Q}$ given by
$$E:x^3+y^3=m_0z^3$$

\noindent
with base point $[1,-1,0]$ and Mordell-Weil rank $r$.\\

\noindent
Observe that any non-torsion point $Q=[x(Q),y(Q),z(Q)]\in E(\mathbb{Q})$ has $z(Q)\neq 0$ because the only point in $E(\mathbb{Q})$ with $z(Q)=0$ is $Q=[1,-1,0]$, the identity element of the Mordell-Weil group.\\

\noindent
 As the rank of $E(\mathbb{Q})$ is $r$, we can choose $r$ independent points $P_1,...,P_r$ from the free part of the group $E(\mathbb{Q})$ and for any point $Q \in E(\mathbb{Q})$, we write $Q=[x(Q),y(Q),z(Q)]$ with $x(Q), y(Q), z(Q) \in \mathbb Z $ and $gcd\left(x(Q),y(Q),z(Q)\right)=1$.\\

\noindent
 Now fix a large positive integer $N$. For each $n=(n_1,...,n_r)\in \mathbb Z^r$ with 
 $1\leq n_i \leq N$, consider the sum
 
 \begin{center}
     $Q_n=n_1P_1+...+n_rP_r$
 \end{center}

\bigskip

\noindent
 which gives $N^r$ distinct points in $E(\mathbb Q)$. Now consider 

 $$
 m=m_0\prod_{n} z(Q_n)^3,
 $$ 
 where the product runs over all $r-$tuples $(n_1,...,n_r)$. Note that, $m\neq 0$ as $z(Q_n)\neq 0$.
\bigskip

\noindent
Hence for each $r$-tuple $n'=(n_1',...,n_r')$, the equation $f(x,y)=x^3+y^3=m$ has the following integral solution
$$(x,y)=\left(x(Q_{n'})\prod_{n\neq n'} z(Q_n)^3 , y(Q_{n'})\prod_{n\neq n'} z(Q_n)^3\right).$$

\noindent
Therefore, we get 
 \begin{center}
      $N_f(m)>N^r$.
 \end{center}
\bigskip

\noindent
We will now use the properties of height function to give an upper bound for $m$ in terms of $N$ and $r$. To do this explicitly, we first transform the elliptic curve $E/\mathbb Q$ into its Weierstrass form and then proceed by using the explicit properties of the height function on that Weierstrass form.

\medskip
\noindent
Consider the following morphism which takes $E/\mathbb Q$ to its Weierstrass form $E'/\mathbb Q$

$$  \Phi : E \rightarrow E'
$$
 defined as

$$ \Phi([x,y,z]) = \begin{cases} \left[12m_0 \, \frac{z}{y+x}, \, 36m_0 \, \frac{y-x}{y+x}, 1\right], &  z\neq 0, \\ [0,1,0], &  z=0, \end{cases} $$

\bigskip
\noindent
where $E': Y^2 = X^3 - 432 \, m_0^2$ denotes the Weierstrass form of $E: x^3+y^3=m_0z^3$.\\

\noindent
Note that, $\Phi$ induces a group homomorphism $\Phi : E(\mathbb Q) \rightarrow E'(\mathbb Q)$ of the corresponding Mordell-Weil groups.

\bigskip
\noindent
Let $Q_n=[x(Q_n),y(Q_n),z(Q_n)]\in E(\mathbb Q)$ be as above. As $z(Q_n)\neq 0$,  we have $$x\left(\Phi(Q_n)\right)=12m_0\frac{z(Q_n)}{y(Q_n)+x(Q_n)}.$$ 

\bigskip
\noindent
For simplicity we write $x , y, z $ for $x(Q_n), y(Q_n)$ and $z(Q_n)$ respectively.\\

\noindent
Hence, 
\begin{align*}
    h_x\left(\Phi(Q_n)\right)&=\log \left(H\left([12m_0\frac{z}{y+x},1]\right)\right)\\
    &= \log \left(H\left([12m_0z,x+y]\right)\right)\\
    &=\log \left(\text{max}\{|\frac{12m_0z}{d}|,|\frac{x+y}{d}|\}\right),
\end{align*}

\noindent
where $d = gcd(12m_0z,x+y)$ \hspace{5Pt} and \hspace{5Pt} $\log \left(\text{max}\{|\frac{12m_0z}{d}|,|\frac{x+y}{d}|\}\right)\geq \log \left(\frac{|12m_0z|}{|d|}\right)$.\\

\noindent
So, 
$$h_x\left(\Phi(Q_n)\right) + \log(|d|)\geq \log (12|m_0|)+ \log(|z|).$$

\bigskip

\noindent
Now, using Lemma \ref{Lemma 1} 
$$h_x\left(\Phi(Q_n)\right) \geq b_{m_0}+ \frac{1}{2}\log(|z|), $$

\noindent
where $b_{m_0}$ is a constant depending on $m_0$. Further, using (e) of Lemma \ref{lemma 2}, we obtain
\bigskip

\begin{center}
    $\log(|z|)\leq 4 \hat{h}(\Phi(Q_n))+c_{m_0}$,
\end{center}

\bigskip
\noindent
where $c_{m_0}$ is another constant depending on $m_0$. \\

\noindent
As $\Phi : E(\mathbb Q) \rightarrow E'(\mathbb Q)$ is group homomorphism,

\bigskip

\begin{center}
    $\hat{h}\left(\Phi(Q_n)\right) = \hat{h}\left(\sum_{i=1}^{r} n_i \Phi(P_i)\right) $.
\end{center}

\bigskip
\noindent
Using (a), (b) and (d) of Lemma \ref{lemma 2}, we deduce the following estimate for $\hat{h}$.

\begin{align*}
    \hat{h}(\Phi(Q_n))\leq (3\times2^{r-1}-2)\max \{ \hat{h}(n_i \Phi(P_i)) \mid 1\leq i  \leq r \}\\
    \leq (3\times2^{r-1}-2)N^2\max \{ \hat{h}(\Phi(P_i)) \mid 1\leq i  \leq r \} 
\end{align*}

\bigskip
\noindent
For ease of notation, write $\hat{h}(\Bar{P}):=\max \{ \hat{h}(\Phi(P_i)) \mid 1\leq i  \leq r \}$, where $\Bar{P}=\Phi(P_i)$ for some $i$.

\bigskip
\noindent
Hence, 
\begin{align*}
    \log (|z(Q_n)|) & \leq 4 (3\times2^{r-1}-2)N^2 \hat{h}(\Bar{P})+c_{m_0}\\
    & \leq (3\times2^{r+1}-7)N^2 \hat{h}(\Bar{P}),
\end{align*}
as we can choose a large $N$ such that $c_{m_0}\leq N^2\hat{h}(\Bar{P})$.

\bigskip
\noindent
Now, we have a total of $N^r$ points $Q_n$ on $E(\mathbb Q)$. Therefore for large $N$,

$$\log|m|=3\sum_{n} \log|z(Q_n)|+\log (|m_0|)\leq (3^2\times2^{r+1}-20) N^{r+2}\hat{h}(\Bar{P}).$$
\noindent
So, we conclude

$$N_f(m)>N^r \geq \frac{1}{(9\times2^{r+1}-20)^{r/r+2}(\hat{h}(\Bar{P}))^{r/r+2}}
(\log |m|)^{r /(r+2)}$$

\bigskip

\noindent
which proves Theorem \ref{Main theorem}, because $N$ can be chosen arbitrarily large, giving infinitely many choices for $m$.

\bigskip

\noindent
The proof of the Corollary \ref{Main result} follows by using the elliptic curve of Theorem \ref{prop 1} in Theorem \ref{Main theorem}. Observe that, using Lemma \ref{lemma 3} on the elliptic curve $Y^2=X^3-432m_0^2$ with $m_0=13293998056584952174157235$, we have 

$$\hat{h}(\Bar{P})\leq 121.767/6 + 76.61/2 + 1.576 = 60.17. $$

\bigskip

\noindent
Now, by putting $r=11$, the required estimate holds for infinitely many integers $m$ and we conclude

 $$N_{f}(m)>4.2\times 10^{-6}(\log |m|)^{11/13}.$$

\bigskip

\section{\bf Concluding remarks}

\medskip

\noindent
It will be interesting to know, if a similar explicit estimate can be obtained for other cubic forms as well. The key idea is to get a result similar to Lemma \ref{Lemma 1} for a general form, then the methods in this article are amenable to give an explicit estimate. 

\section*{Acknowledgments}

\noindent
I thank Dr. Anup Dixit and Sushant Kala for their invaluable support, guidance and numerous fruitful discussions. I also thank Prof. C. L. Stewart for letting me know of the result in \cite{c7}. I am thankful to The Institute of Mathematical Sciences, Chennai for supporting my stay as a visiting student.

\end{document}